\documentclass{article}

\usepackage[preprint,nonatbib]{neurips_2021}

\usepackage[utf8]{inputenc} 
\usepackage[T1]{fontenc}    
\usepackage{hyperref}       
\usepackage{url}            
\usepackage{booktabs}       
\usepackage{amsfonts}       
\usepackage{nicefrac}       
\usepackage{microtype}      
\usepackage{xcolor}         
\usepackage{doi} 
\usepackage{mathtools}

\newcommand\myeq{\stackrel{\mathclap{\normalfont\mbox{def}}}{=}}

\title{Unifying Width-Reduced Methods for Quasi-Self-Concordant Optimization}

\author{%
  Deeksha Adil \\
University of Toronto\\
  \texttt{deeksha@cs.toronto.edu} \\
   \And
  Brian Bullins \\
  TTI Chicago \\
   \texttt{bbullins@ttic.edu} \\
   \And
   Sushant Sachdeva \\
  University of Toronto \\
   \texttt{sachdeva@cs.toronto.edu} \\
}

\def\showauthornotes{0}
\def\showdraftbox{0}

\newcommand{\deeksha}{\Authornote{Deeksha}}

\usepackage{amstext,bbm,xspace, amsmath, amssymb, graphicx, url, amsthm}
\usepackage{color,enumerate,latexsym,bm,verbatim,textcomp}
\usepackage[small]{caption}
\usepackage{comment} 
\usepackage{algorithm} 
\usepackage[noend]{algpseudocode}
\usepackage{thmtools}
\usepackage{thm-restate}
\usepackage{fancybox}
\usepackage[shortlabels]{enumitem}
\usepackage{commath}
\usepackage{mdframed}

\newcommand{\defeq}{\stackrel{\textup{def}}{=}}

\newtheorem{theorem}{Theorem}[section]

\newtheorem{lemma}[theorem]{Lemma}

\newtheorem{definition}[theorem]{Definition}

\def\abs#1{\left| #1 \right|}
\renewcommand{\norm}[1]{\ensuremath{\left\lVert #1 \right\rVert}}

\newcommand\rea{\mathbb R}

\newcommand{\marginlabel}[1]%
{\mbox{}\marginpar{\it{\raggedleft\hspace{0pt}#1}}}

\definecolor{Mygray}{gray}{0.8}

 \ifcsname ifcommentflag\endcsname\else
  \expandafter\let\csname ifcommentflag\expandafter\endcsname
                  \csname iffalse\endcsname
\fi

\ifnum\showauthornotes=1
\newcommand{\Authornote}[2]{{\sf\small\color{red}{[#1: #2]}}}
\newcommand{\Authoredit}[2]{{\sf\small\color{red}{[#1]}\color{blue}{#2}}}
\newcommand{\Authorcomment}[2]{{\sf \small\color{gray}{[#1: #2]}}}
\newcommand{\Authorfnote}[2]{\footnote{\color{red}{#1: #2}}}
\newcommand{\Authorfixme}[1]{\Authornote{#1}{\textbf{??}}}
\newcommand{\Authormarginmark}[1]{\marginpar{\textcolor{red}{\fbox{
#1:!}}}}
\else
\newcommand{\Authornote}[2]{}
\newcommand{\Authoredit}[2]{}
\newcommand{\Authorcomment}[2]{}
\newcommand{\Authorfnote}[2]{}
\newcommand{\Authorfixme}[1]{}
\newcommand{\Authormarginmark}[1]{}
\fi

\newlength{\pgmtab}  
\let\originalleft\left
\let\originalright\right
\renewcommand{\left}{\mathopen{}\mathclose\bgroup\originalleft}
  \renewcommand{\right}{\aftergroup\egroup\originalright}

\def\defeq{\stackrel{\mathrm{def}}{=}}

\newcommand\bb{\boldsymbol{\mathit{b}}}
\newcommand\cc{\boldsymbol{\mathit{c}}}
\newcommand\dd{\boldsymbol{\mathit{d}}}

\newcommand\ff{\boldsymbol{\mathit{f}}}

\newcommand\rr{\boldsymbol{\mathit{r}}}

\newcommand\ww{\boldsymbol{\mathit{w}}}
\newcommand\yy{\boldsymbol{\mathit{y}}}
\newcommand\zz{\boldsymbol{\mathit{z}}}
\newcommand\xx{\boldsymbol{\mathit{x}}}

\newcommand\xxbar{\overline{\boldsymbol{\mathit{x}}}}
\newcommand\xxtil{\widetilde{\boldsymbol{\mathit{x}}}}

\renewcommand\AA{\boldsymbol{\mathit{A}}}

\newcommand\CC{\boldsymbol{\mathit{C}}}

\newcommand\II{\boldsymbol{\mathit{I}}}

\newcommand\PP{\boldsymbol{\mathit{P}}}
\newcommand\QQ{\boldsymbol{\mathit{Q}}}
\newcommand\RR{\boldsymbol{\mathit{R}}}
\renewcommand\SS{\boldsymbol{\mathit{S}}}
\newcommand\UU{\boldsymbol{\mathit{U}}}

\newcommand\VV{\boldsymbol{\mathit{V}}}
\newcommand\XX{\boldsymbol{\mathit{X}}}

\newcommand\Otil{\widetilde{O}}

\newcommand\Dtil{{\widetilde{{\Delta}}}}

\newcommand\Dopt{{{{\Delta^{\star}}}}}

 {
	\begin{enumerate}}{\end{enumerate}}

\ifnum\showdraftbox=1

\else

\fi

\begin{document}

\maketitle

\begin{abstract}
We provide several algorithms for constrained optimization of a large class of convex problems, including softmax, $\ell_p$ regression, and logistic regression. Central to our approach is the notion of width reduction, a technique which has proven immensely useful in the context of maximum flow [Christiano et al., STOC'11] and, more recently, $\ell_p$ regression [Adil et al., SODA'19], in terms of improving the iteration complexity from $O(m^{1/2})$ to $\tilde{O}(m^{1/3})$, where $m$ is the number of rows of the design matrix, and where each iteration amounts to a linear system solve. However, a considerable drawback is that these methods require both problem-specific potentials and individually tailored analyses.

As our main contribution, we initiate a new direction of study by presenting the first \emph{unified} approach to achieving $m^{1/3}$-type rates. Notably, our method goes beyond these previously considered problems to more broadly capture \emph{quasi-self-concordant} losses, a class which has recently generated much interest and includes the well-studied problem of logistic regression, among others. In order to do so, we develop a unified width reduction method for carefully handling these losses based on a more general set of potentials. Additionally, we directly achieve $m^{1/3}$-type rates in the constrained setting without the need for any explicit acceleration schemes, thus naturally complementing recent work based on a ball-oracle approach [Carmon et al., NeurIPS'20].
\end{abstract}


\section{Introduction}
We study a class of constrained optimization problems of the following form:
\begin{equation}\label{eq:Problem}
\min_{\AA\xx=\bb} \sum_i \ff\left((\PP\xx)_i\right)
\end{equation}
for convex $\ff : \rea \rightarrow \rea$, where
$\AA\in \mathbb{R}^{d \times n}, \bb \in \mathbb{R}^{d}, \PP \in
\mathbb{R}^{m \times n},$ with $ d \leq n \le m$. Specifically, we are
interested in the case where $\ff$ satisfies a certain higher-order
smoothness-like condition known as $M$-quasi-self-concordance
(q.s.c.), i.e., $|\ff'''(\xx)| \leq M\ff''(\xx)$ for all
$\xx \in \rea$. Several problems of significant interest in machine
learning and numerical methods meet this condition, including logistic
regression \cite{bach2010self,karimireddy2018global}, as well as
softmax (often used to approximate $\ell_\infty$ regression)
\cite{nesterov2005smooth, CKMST, ene2019improved, bullins2020highly}
and (regularized) $\ell_p$ regression \cite{bubeck2018homotopy,
  adil2019iterative}.

A very useful optimization technique, first introduced by \cite{CKMST}
for faster approximate maximum flow and later by
\cite{chin2013runtime} for regression, is that of width reduction,
whereby they used it to improve the iteration complexity dependence on
$m$, the number of rows of the design matrix from $O(m^{1/2})$ to
$\tilde{O}(m^{1/3})$, and where each iteration requires a linear
system solve. Later work by \cite{adil2019iterative} for high-accuracy
$\ell_p$ regression, building on an $O(m^{1/2})$-iteration result from
\cite{bubeck2018homotopy}, again showed how width reduction could lead
to improved $\tilde{O}(m^{1/3})$-iteration algorithms. As a drawback, however,
these approaches rely on potential methods and analyses specifically
tailored to each problem.

Building on these results, we present the first \emph{unified}
approach to achieving $m^{1/3}$-type rates, at the heart of which lies
a more general width reduction scheme. Notably, our method goes beyond
these previously considered problems to capture
\emph{quasi-self-concordant} losses, thereby further including
well-studied problems such as logistic regression, among others. By doing so, we directly achieve $m^{1/3}$-type rates in the constrained setting without relying on explicit acceleration schemes \cite{monteiro2013accelerated}, thus complementing recent work based on a ball-oracle approach \cite{CJJJLST}. We additionally note that, given the ways in which our results achieve improvements similar to those of \cite{CJJJLST}, we believe our work hints at a deeper, though to our knowledge not yet fully understood, connection between the techniques of width reduction and Monteiro-Svaiter acceleration.

\subsection{Main Results and Applications}
We first present in Section~\ref{sec:CrudeOracle} a width-reduced
method for obtaining a crude approximation to \eqref{eq:Problem} for
quasi-self-concordant $\ff$. At a high level, our algorithm returns an
approximate solution $\tilde{\xx}$ that both satisfies the linear
constraints and is bounded in $\ell_\infty$-norm by $O(R)$, where $R$
is a bound on the norm of the optimal solution. Following from
Theorem~\ref{lem:runtime-qsc}, the result below shows how, for the
problem of minimizing softmax (parameterized by $\nu>0$), i.e., $\textrm{smax}_{\nu}(\PP\xx) = \nu \log \left(\sum_i e^{\frac{(\PP\xx)_i}{\nu}}\right)$, we can bound
the norm of the solution by $(1+\nu)R$.
\begin{restatable}{theorem}{smax_intro}{\label{thm:smax_intro}}
Let $\xx^{\star}$ denote the optimum of $\min_{\AA\xx=\bb}\emph{\textrm{smax}}_{\nu}(\PP\xx)$. Algorithm \ref{alg:MainAlgo} when applied to the function $\ff(\PP\xx) = \sum_i e^{\frac{(\PP\xx)_i}{\nu}}$ with $\epsilon = \nu$, returns $\xxtil$ such that $\AA\xxtil = \bb$, and 
\[
\emph{\textrm{smax}}_{\nu}(\PP\xxtil) \leq (1+\Otil(\nu))\emph{\textrm{smax}}_{\nu}(\PP\xx^{\star}),
\]
in at most $\Otil(m^{1/3}\nu^{-5/3})$ calls to a linear system solver. 
\end{restatable}

As a consequence of Theorem~\ref{thm:smax_intro} when taking $\nu = \Omega\left(\epsilon / \log^{O(1)}(m)\right)$, we have by Theorem~\ref{thm:inftyreg} a $(1+\epsilon)$ approximate solution to the problem of $\ell_\infty$ regression with $\tilde{O}(m^{1/3}\epsilon^{-5/3})$ calls to a linear system solver.

Further, we show the following result which can use the approximate
solution returned by Theorem~\ref{thm:smax_intro} as an initial point
for achieving a high-accuracy solution.
We also present in Appendix \ref{app:gsc} a natural extension of our
results to minimizing general-self-concordant (g.s.c.) functions.
\begin{theorem}
  \label{thm:main-eps-informal}
  For $M$-q.s.c. $\ff$, $\epsilon >0$, and $\xx^{(0)}$ such that
  $\AA\xx^{(0)} = \bb$ and $\|\xx^{(0)}\|_{\infty} \leq R$, Algorithm
  \ref{alg:eps} finds $\xxtil$ such that $\AA\xxtil = \bb$ and
  $\ff(\xxtil) \leq \epsilon + \ff(\xx^{\star}) $ in
  $\Otil\left(MR m^{1/3}\log (MR) \log
    \left(\frac{\ff(\xx^{(0)})-\ff(\xx^{\star})}{\epsilon}\right)\right)$ calls to a linear system
  solver.
\end{theorem}

Resulting from the theorem above, as detailed in Section~\ref{sec:SpecialFunc}, are guarantees given by Theorems \ref{thm:lpreg} and \ref{thm:logisticRuntime} which establish convergence rates of $\widetilde{O}(p^2 \mu^{-1/(p-2)}m^{1/3}R)$ and $\widetilde{O}(m^{1/3}R)$, respectively, for $\mu$-regularized $\ell_p$ regression and logistic regression. We emphasize that the latter is, to our knowledge, the first such use of width reduction for directly solving constrained logistic regression problems.

\subsection{Related Works}
\paragraph{Quasi-self-concordance and higher-order smoothness.} \cite{bach2010self} showed how to analyze Newton's method for quasi-self-concordant functions, with an emphasis on its application to logistic regression. Later, notions of local, or Hessian, stability which follow from quasi-self-concordance gave rise to methods with better dependence on various conditioning parameters \cite{karimireddy2018global, CJJJLST}. In the work by \cite{karimireddy2018global}, the authors show how a trust-region-based Newton method \cite{nocedal2006numerical} achieves linear convergence for locally stable functions without requiring, e.g., strong convexity. Meanwhile, after noting that quasi-self-concordance implies Hessian stability, \cite{CJJJLST} further improve the dependence on the distance to the optimum by leveraging Monteiro-Svaiter acceleration \cite{monteiro2013accelerated}, which has proven useful in the context of near-optimal methods for higher-order acceleration \cite{arjevani2019oracle, gasnikov2019near}. However, in general these methods, which assume higher-order smoothness, require access to an oracle which minimizes a higher-order Taylor expansion model, though in some cases this may be relaxed to requiring linear system solves \cite{bullins2020highly}.

\paragraph{Width reduction and $\ell_p$ regression.} The technique of width-reduction first came to prominence in seminal work by \cite{CKMST} for achieving faster approximate maximum flow, being the first to achieve an improved $m^{1/3}$ dependence. At a high level, the idea behind the approach is to solve a sequence of weighted $\ell_2$-minimizing flow problems, whereby at each iteration one of two cases occurs: either the proposed step is added to the current solution (a "flow" step) along with the weights, or else there exist some set of coordinates that exceed a certain threshold, and so their weights are updated accordingly (a "width reduction" step). Several works have since adapted this approach to regression problems \cite{chin2013runtime, adil2019iterative, ene2019improved, ABKS21} and matrix scaling \cite{allen2017much}.

In addition to their importance in machine learning, regression
methods capture several fundamental problems in scientific computing
and signal processing. A recent line of work initiated by
\cite{bubeck2018homotopy} showed how to attain high-accuracy solutions
for $\ell_p$ regression using $O_p(m^{1/2-1/p})$ linear system solves,
thus going beyond what is achievable via self-concordance. Building on
this work, \cite{adil2019iterative} showed how width reduction could
be applied to this setting to achieve, as in the case of approximate
maximum flow \cite{CKMST}, a similar improvement from $O_p(m^{1/2})$
to $O_p(m^{1/3})$ (for $p \rightarrow \infty$). Further developments
by \cite{kyng2019flows, adil2020faster} for graph problems showed almost-linear time solutions for $\ell_p$ regression for $p \approx \sqrt{\log(n)}$ which have since been a critical part of recent advances in high-accuracy maximum flow on unit-capacity graphs \cite{liu2020faster, kathuria2020unit}.

\paragraph{Accelerated methods.} Recent developments by \cite{CJJJLST} have shown several advantages that arise in the case of unconstrained minimization for $L$-smooth $M$-quasi-self-concordant problems. By considering a certain ball oracle method (whereby each call to the oracle returns the minimizer of the function inside an $\ell_2$ ball of radius $r$), \cite{CJJJLST} implement an accelerated scheme which returns a solution to the unconstrained smooth convex minimization problem in $(R/r)^{2/3}$ calls to the oracle, where $R$ is the initial $\ell_2$-norm distance to the optimum, and they further show a matching lower bound under this oracle model. While this approach transfers its difficulty to implementing the oracle, a key insight from their work involves showing this can be done efficiently for smooth quasi-self-concordant functions when $r$ is sufficiently small, where the allowed size depends on the quasi-self-concordance parameter. One limitation to the results of \cite{CJJJLST} is that they apply directly to \emph{unconstrained} optimization problems and require the function to be smooth, and so we complement these results in the quasi-self-concordant setting by establishing comparable rates for directly optimizing a large class of \emph{constrained} convex problems without requiring smoothness of the function.

\subsection{Outline of the Paper}
After establishing the potential functions at the heart of our width reduction techniques, we present in Section~\ref{sec:CrudeOracle} our oracle for roughly approximating a solution to problem \eqref{eq:Problem}. We then show in Section~\ref{sec:boost} how we may attain a high-accuracy solution by using the crude approximation as a starting point. Here, the key idea is to considering a sequence of optimization problems inside $\ell_\infty$ balls of manageable size, similar to \cite{cohen2017matrix, CJJJLST}. As in the case of the crude oracle, our primary advantage comes from carefully handling a pair of coupled potentials which are amenable to the large class of quasi-self-concordant problems, and in Section~\ref{sec:SpecialFunc} we further show how our results may be applied to several problems of interest, including logistic and $\ell_p$ regression.



\section{Preliminaries}\label{sec:prelims}

\paragraph{Notation:} We use boldface lowercase letters to denote vectors or functions and boldface uppercase letters for matrices. Scalars are non-bold letters. Our functions are univariate, and we overload function notation to act on a  vector coordinate-wise, i.e. $\ff(\xx) = \sum_i \ff(\xx_i)$. The notation $\xx \geq \yy$ for vectors refers to entry-wise inequality. Refer to the algorithm boxes for definitions of certain algorithm specific parameters that appear in lemma and theorem statements.

\subsection{Quasi-Self-Concordance}
\begin{definition}[g.s.c. and q.s.c.]\label{def:gsc}
  Let $\ff :\mathbb{R} \rightarrow \mathbb{R}$ be a thrice differentiable function with continuous third derivative, and let $\nu >0$ and $M>0$. We say that $\ff$ is $(M,\nu)$-general-self-concordant (g.s.c.) if
  \[
  \forall x, \quad  |\ff'''(x)|\leq M \ff''(x)^{\frac{\nu}{2}}.
  \]
  When $\nu= 2$, we have the following condition:
  \[
  \forall x, \quad  |\ff'''(x)|\leq M \ff''(x),
  \]
  and we call such functions $M$-quasi-self-concordant (q.s.c.).
\end{definition}

\subsection{Problem}

Recall that we are solving the following problem:
\begin{equation*}
  \min_{\AA\xx=\bb} \sum_i \ff\left((\PP\xx)_i\right),
\end{equation*}
where $\AA\in \mathbb{R}^{d \times n}, \bb \in \mathbb{R}^{d}, \PP \in \mathbb{R}^{m \times n}, d \leq n$, and $m\geq n$,
 and such that $\ff$ is convex, $M$-q.s.c. and, for $\ww\geq \ww_0 \geq 0$, $\ff''(\ww)$ is coordinate-wise monotonic. We can ignore the case when $\ff''$ is constant since that corresponds to a quadratic problem which we know how to solve directly via linear system solves.

\subsubsection*{Assumptions on the Optimum $\xx^{\star}$}
We assume that $R \in \mathbb{R}_{>0}$ is such that the optimum $\xx^{\star} \defeq \arg\min_{\AA\xx=\bb} \ff(\PP\xx)$ satisfies 
\begin{equation}\label{eq:ass-opt}
 \|\PP\xx^{\star}\|_{\infty} \leq R.
\end{equation}

We now define the potentials that we track in the algorithm.

\subsection{Potentials}

\begin{definition}[Dual Potential]\label{def:dual}
For a weights vector $\ww \in \mathbb{R}_{\geq 0}^m$, we define a potential 
\[
\Phi(\ww) = \sum_i \Phi(\ww_i) =  \sum_i \ff''(\ww_i).
\]
\end{definition}

We also define the following corresponding potential, which gives rise to the linear regression problem that we will need to solve at each step of our algorithm.
\begin{definition}[Resistances and Corresponding Potential]\label{def:res}
For a weights vector $\ww \in \mathbb{R}_{\geq 0}^m$ and $\epsilon>0$, define resistances $\rr \in \mathbb{R}_{\geq 0}^m$ and a corresponding potential $\Psi$ as,
\[
\rr_i = \frac{1}{R^2}\left(\ff''(\ww_i) + \frac{\epsilon\Phi(\ww)}{m}\right), 
\]
\[
\Psi(\rr)= \min_{\AA\Delta = \bb} \sum_i \rr_i (\PP\Delta)_i^2.
\]
\end{definition}

We have the following relation between our two potentials $\Phi$ and $\Psi$.

\begin{restatable}{lemma}{PhiPsi}\label{lem:Relate-Psi-Phi}
For $\epsilon>0$, resistances $\rr$ (Definition \ref{def:res}), with corresponding weights $\ww$, we have
\[
\Psi(\rr) \leq  (1+\epsilon) \Phi(\ww).
\]
In addition, letting $\|\PP\|_{\min} = \min_{\AA\xx=\bb}\|\PP\xx\|_2$ and $\|\AA\|$ denote the operator norm of $\AA$, we have
\[
\Psi(\rr) \geq \frac{\epsilon\Phi(\ww)}{m R^2}  \frac{\|\PP\|_{\min}^2\|\bb\|_2^2}{\|\AA\|^2}\  \myeq \ \Phi(\ww) L.
\]
\end{restatable}



\section{Algorithm and Analysis for a Crude Solution for Q.S.C. Functions}\label{sec:CrudeOracle}

In this section, we give an algorithm for solving Problem \eqref{eq:Problem} to a crude approximation; namely, we return a solution $\xxtil$ such that $\AA\xxtil = \bb$, i.e., it satisfies the subspace constraints, and $\|\PP\xxtil\|_{\infty}$ is bounded. We will later see in our applications how this translates into a constant or polynomial approximation guarantee to the function value for some functions. In the next section we will see how we can use the guarantees of the solution returned as a starting solution and boost it to an $\epsilon$ approximate solution. 

Our algorithm is based on combining a multiplicative weight update (MWU) scheme with width reduction. Though such algorithms have so far only been used for $\ell_p$-regression, $p =1$ or $p\in[2,\infty]$, here we are able to extend the analysis to q.s.c. functions, while also providing a unified analysis for the known cases of $\ell_p$-regression (refer to Section \ref{sec:SpecialFunc} to see how we apply this algorithm to these instances). We note that we can extend this analysis to other general-self-concordant functions, and we have deferred these cases to the appendix.

\subsection{Algorithm and Analysis}

Our proof relies on tracking two potentials, $\Psi$ (Definition \ref{def:res}) and $\Phi$ (Definition \ref{def:dual}) that depend on the weights. We first show how these potentials change with weight updates corresponding to a flow step and a width reduction step in the algorithm. We next show that if our algorithm runs for at most $K=\Otil(m^{1/3})$ width reduction steps, then after $T = \Otil(m^{1/3})$ flow steps we can bound $\Phi$. Further, using the relation between $\Phi$ and $\Psi$ (Lemma \ref{lem:Relate-Psi-Phi}) and appropriately chosen parameters, we show that we cannot have more than $K$ width reduction steps. The key part of the analysis lies in the growth of $\Phi$ with respect to both flow and width steps.

\begin{algorithm}
\caption{Width-Reduced Algorithm for $M$-q.s.c. Functions}\label{alg:MainAlgo}
 \begin{algorithmic}[1]
\Procedure{QSC-MWU}{$\AA, \bb, \PP,M,R,\epsilon$}
\State $\xx^{(0,0)} = 0, \ww^{(0,0)} = \ww_0$ ($\Phi'(\ww)$ monotonic for $\ww\geq \ww_0$, $\Phi(\ww_0)>0$)
\State  $\tau \leftarrow \widetilde{\Theta}\left( m^{1/3}\epsilon^{-2/3}\right)$
\State $\alpha \leftarrow \widetilde{\Theta}\left(m^{-1/3}M^{-1} \epsilon^{2/3}\right)$
\State $ t = 0,k=0, \quad T = \alpha^{-1}M^{-1}\epsilon^{-1} = \widetilde{\Theta}(m^{1/3}\epsilon^{-5/3})$
\While{ $t \leq T$}
	\State $\rr^{(t,k)}_i \leftarrow \frac{1}{R^2}\left(\ff''(\ww^{(t,k)}_i) + \frac{\epsilon\Phi(\ww^{(t,k)})}{m}\right)$ \Comment{Resistances}\deeksha{add t,k}
	 \State $\Dtil \leftarrow \arg \min_{\AA\Delta = \bb}  \quad  \sum_i \rr_i (\PP\Delta)_i^2$ \Comment{Oracle}
	 \If{$\norm{\PP\Dtil}_{\infty} \leq R \tau $}\Comment{Flow Step}
		\State $\xx^{(t+1,k)} \leftarrow \xx^{(t,k)} +  \Dtil$
		\State $\ww^{(t+1,k)} \leftarrow \ww^{(t,k)} + \frac{\epsilon\alpha}{R}|\PP\Dtil|$
		\State $ t \leftarrow t+1$
	\Else 
		 \For{Indices $i$ such that $|\PP\Dtil|_i \geq R \tau$}\Comment{Width Reduction}
		 	\If{$\ff''$ is non-decreasing in $\ww$}\footnotemark
			 	\State $\ww^{(t,k+1)}$ is such that $\rr^{(t,k+1)}_i \leftarrow (1+\epsilon) \rr^{(t,k)}_i$ 
			 \Else
			 	\State $\ww^{(t,k+1)}$ is such that $\rr^{(t,k+1)}_i \leftarrow \frac{1}{1+\epsilon} \rr^{(t,k)}_i$ 
			\EndIf
		\EndFor
		\State $k \leftarrow k+1$
	\EndIf
\EndWhile
\State\Return $\xx^{(T,k)}/T$
\EndProcedure 
 \end{algorithmic}
\end{algorithm}
\footnotetext{We will see later how such weight/resistance changes can be realized for some special cases.}

\subsubsection*{Changes in $\Psi$ and $\Phi$}

\begin{restatable}{lemma}{Width}\label{lem:Width}
Let $\Psi$ be as defined in \ref{def:res}. After $t$ flow steps and $k$ width reduction steps, we have, 
\begin{align*}
  &\Psi(\rr^{(t,k)}) \geq \Psi(\rr^{(0,0)}) \left(1 +  \frac{\epsilon^2 \tau^2}{(1+\epsilon)^2m}\right)^k && \text{ if $\ff''$ non-decreasing in $\ww$,} \\
  &\Psi(\rr^{(t,k)}) \leq \Psi(\rr^{(0,0)})\left(1 - \frac{\epsilon^2\tau^2}{2(1+\epsilon)^2 m}\right)^k && \text{if $\ff''$ non-increasing in $\ww$}.
\end{align*}
\end{restatable}

\begin{restatable}{lemma}{Flow}\label{lem:flow}
Suppose $\ff$ is $M$-q.s.c. Let $\alpha$ and $\tau$ be such that $\alpha \tau \leq M^{-1}$. After $t$ flow steps and $k$ width reduction steps, our potential $\Phi$ satisfies
\begin{align*}
  &\Phi(\ww^{(t,k)}) \leq  \left(1 + \epsilon(1+\epsilon)^2 \alpha M \right)^t \left(1 + \epsilon(1+\epsilon)\tau^{-1} \right)^k \Phi(\ww_0) && \text{ if $\ff''$ non-decreasing in $\ww$,} \\
  &\Phi(\ww^{(t,k)}) \geq  \left(1 - \epsilon(1+\epsilon)^2 \alpha M \right)^t \left(1 - \epsilon(1+\epsilon)\tau^{-1} \right)^k \Phi(\ww_0) && \text{if $\ff''$ non-increasing in $\ww$}.
\end{align*}
\end{restatable}

\subsubsection*{Runtime Bound}

\begin{theorem}\label{lem:runtime-qsc}
Let $\epsilon>0$, $\ff$ be $M$-q.s.c. After $T \leq  \frac{\alpha^{-1}}{M \epsilon}= \widetilde{\Theta}(m^{1/3}\epsilon^{-5/3})$ flow steps and $K \leq \tau = \widetilde{\Theta}(m^{1/3} \epsilon^{-2/3})$ width reduction steps, Algorithm \ref{alg:MainAlgo} returns $\xxtil$ such that $\AA\xxtil = \bb$, $\|\PP\xxtil\|_{\infty} \leq R M \|\ww^{(T,K)}\|_{\infty}$, where $\ww^{(T,K)}$ is the final weights vector that satisfies: 
\begin{align*}
  &\Phi(\ww^{(T,K)}) \leq \Phi(\ww_0)e^{1+4\epsilon} && \text{ if $\ff''$ is non-decreasing in $\ww$,} \\
  &\Phi(\ww^{(T,K)}) \geq  \Phi(\ww_0) e^{-(1+4\epsilon)} && \text{if $\ff''$ is non-increasing in $\ww$}.
\end{align*}
\end{theorem} 
\begin{proof}
We show the case when $\ff''$ is a non-decreasing function. The other case follows similarly. We set,
\[
\tau \leftarrow \widetilde{\Theta}\left( m^{1/3}\epsilon^{-2/3}\right) \quad  \alpha \leftarrow \widetilde{\Theta}\left(m^{-1/3}M^{-1}\epsilon^{2/3}\right).
\]
After $T = \frac{\alpha^{-1}}{M\epsilon}$ flow steps and $K = \tau  $ width reduction steps, from Lemma  \ref{lem:flow}, we have,
\begin{align*}
\Phi(\ww^{(T,K)}) &\leq \left(1 + \epsilon(1+\epsilon)^2 \alpha M \right)^T \left(1 + \epsilon(1+\epsilon)\tau^{-1} \right)^K \Phi(\ww_0)\\
& \leq \Phi(\ww_0) e^{ \epsilon(1+\epsilon)^2 \alpha M T + \epsilon(1+\epsilon) \tau^{-1} K} \leq \Phi(\ww_0) e^{(1+4\epsilon)}.
\end{align*}

We now show we cannot have more width steps. Throughout the algorithm, we have $\Phi(\ww^{(t,k)}) \leq  \Phi(\ww_0) e^{1+4\epsilon}$. From Lemma \ref{lem:Relate-Psi-Phi} we always have $\Psi(\rr^{(0,0)})\geq \Phi(\ww_0) L$ and $\Psi(\rr^{(T,K)}) \leq (1+\epsilon)\Phi(\ww^{(T,K)})  \leq (1+\epsilon)e^{1+4\epsilon}\Phi(\ww_0)$. Thus, from Lemma \ref{lem:Width}, we must have,
\[
(1+\epsilon)e^{1+4\epsilon}\Phi(\ww_0)\geq L \Phi(\ww_0)\left(1 + \frac{\epsilon^2 \tau^2}{(1+\epsilon)^2m}\right)^K,
\]
From the definition of $\tau$, we note that $K$ has to be less than $\tau$ for the above bound to be satisfied. Next, let $\Dtil^{(t)}$ denote the solution of our oracle at iteration $t$ of the flow step. From the $\xx$ and $\ww$ update in the algorithm,
\[
\abs{\PP\xxtil}  = \abs{\sum_t \PP\Dtil^{(t)}}\epsilon \alpha M \leq \frac{\epsilon \alpha}{R}\sum_t \abs{\PP\Dtil^{(t)}} RM\leq \ww^{(T,K)} RM.
\]
This concludes our proof.
\end{proof}



\section{Boosting to a High-Accuracy Solution for Q.S.C. Functions} \label{sec:boost}

In this section, we give a width-reduced multiplicative weights update algorithm that, given a starting
solution $\xx^{(0)}$ satisfying $\|\xx^{(0)}\|_{\infty} \leq R$ and
$\AA\xx^{(0)} = \bb$, finds $\xxtil$ such that $\AA\xxtil = \bb$ and
$\ff(\xxtil) \leq (1+\epsilon)\ff(\xx^{\star})$ for any q.s.c. function $\ff$. We would mention here that for the algorithms in this section, it is key that we
have a starting solution that satisfies our subspace constraints and has $\ell_{\infty}$-norm bounded by $R$. Thus, the algorithms here may be of independent interest if such a starting solution is available. We can otherwise
use Algorithm \ref{alg:MainAlgo} with $\epsilon = 1$ to obtain such a
solution.

For any $\xx$, we define a residual problem, and we show how it is
sufficient to solve the residual problem approximately
$\log(\epsilon^{-1})$ times to obtain our high-accuracy solution. Similar approaches have been applied to specific functions such as softmax \cite{allen2017much} and $\ell_p$-regression \cite{adil2019iterative}. We unify these approaches and give a version that works for any q.s.c. function.

We further note that our residual problem is to optimize a simple quadratic objective inside an $\ell_{\infty}$ box. The difficulty lies in solving such $\ell_{\infty}$ box constraints fast. We use a binary search followed by a width-reduced multiplicative weights routine analogous to \cite{CKMST} to solve our residual problem.
\begin{definition}[Residual Problem]
 We define the residual objective at any $\xx$ satisfying $ \|\PP\xx\|_{\infty}\leq R$ as
 \[
 res(\Delta) = \nabla \ff(\xx)^{\top} \PP\Delta  - e^{-1} (\PP\Delta)^{\top}\nabla^2 \ff(\xx) \PP\Delta,
\]
and the residual problem as
\begin{align}\label{eq:res-prob-red}
   \begin{aligned}
     &\max_{\Delta}\quad res(\Delta)\\
     s.t. &\quad\AA\Delta = 0, \quad and \quad \norm{\PP\Delta - \zz}_{\infty} \leq \frac{1}{2M}.
   \end{aligned}
 \end{align}
Here, $\zz$ is a vector that depends on $\xx$, and is defined as
\[
 \zz_i = \begin{cases}
 \left(-\frac{1}{2M} + R + (\PP\xx)_i \right) \in [-\frac{1}{2M},0)], & \text{if } (\PP\xx)_i -\frac{1}{2M} <-R\\
 \left(-R  + (\PP\xx)_i +\frac{1}{2M} \right) \in (0,\frac{1}{2M}], & \text{if } (\PP\xx)_i + \frac{1}{2M} >R \\
 0, & \text{otherwise.}
 \end{cases}
 \]
 \end{definition}

 We note that any solution $\Delta$ satisfying the above box
 constraint satisfies $\norm{\PP\Delta}_{\infty} \leq M^{-1}$ and
 $ \norm{\PP\xx - e^{-2}\PP\Delta}_{\infty}\leq R$.

\begin{restatable}{lemma}{Iterative}[Iterative Refinement]\label{lem:iterative-ref}
 Let $\ff$ be $M$-q.s.c. and $\Dtil^{(t)}$ a $\kappa$-approximate solution to the residual
  problem at $\xx^{(t)}$ (Problem \eqref{eq:res-prob-red}). Starting from $\xx^{(0)}$ such that
  $\AA\xx^{(0)} = \bb$, $\|\xx^{(0)}\|_{\infty} \leq R$, and iterating as
  $\xx^{(t+1)} = \xx^{(t)} - e^{-2}\Dtil^{(t)}$, after at
  most
  $O\left(\kappa MR
    \log\left(\frac{\ff(\xx^{(0)})-\ff(\xx^{\star})}{\epsilon}\right)\right)$
  iterations we get $\xx$ such that $\AA\xx = \bb$ and 
  $\ff(\xx) \leq \ff(\xx^{\star})+\epsilon$.
\end{restatable}

\subsection{Approximately Solving the Residual Problem}
\subsubsection*{Binary Search}

\begin{restatable}{lemma}{binaryNu}\label{lem:res-opt}
Let $\nu$ be such that $\ff(\xx^{(t)}) - \ff(\xx^{\star}) \in (\nu/2,\nu]$ and $\Dopt$ denote the optimum of the residual problem at $\xx^{(t)}$. Then, $res(\Dopt) \in \left(\frac{\nu}{8 MR}, e^{2} \nu\right]$.
\end{restatable}

From the above lemma we may do a binary search in the range $\left(\frac{\nu}{8 MR}, e^{2}\nu\right]$. Let us start with the assumption that the residual problem has a solution between $(\zeta/2,\zeta]$. 

\begin{restatable}{lemma}{binaryZeta}\label{lem:binary}
Let $\zeta$ be such that $res(\Dopt) \in (\zeta/2,\zeta]$ and $\Dopt$ the optimum of the residual problem. Then, $(\PP\Dopt)^{\top}\nabla^2 \ff(\xx) \PP\Dopt \leq e \cdot \zeta$.
\end{restatable}

\subsubsection*{Using Width Reduction}

\begin{algorithm}
\caption{Boosting to $\epsilon$-approximation}\label{alg:eps}
 \begin{algorithmic}[1]
 \Procedure{QSC-min}{($\AA, \bb, \PP, \xx_0,M,\epsilon$) such that $\AA\xx_0 = \bb$, $\|\xx_0-\xx^{\star}\|_{\infty}\leq 2R$}
\State $\xx^{(0)} = \xx_0, \tau \leftarrow m^{1/3},\alpha \leftarrow m^{-1/3}$
\For{$i \leq O(MR \log \epsilon^{-1})$}
	 \For{$\nu \in \left(\epsilon, \ff(\xx)\right] $} \Comment{Decrease $\nu$ by $2$ in each iteration}
		 \For{$\zeta \in \left(\frac{\nu}{8 MR},e^{2}\nu\right]$}\Comment{Decrease $\zeta$ by $2$ in each iteration}
			\State $\yy_{\zeta,\nu} \leftarrow MWU(\AA,\PP,\xx^{(i)},M,\zeta) $
		\EndFor
	\EndFor
\EndFor
\State $\xx^{(i+1)} \leftarrow \xx^{(i)} - e^{-2}\arg\min_{\yy_{\zeta,\nu}} \ff(\xx- e^{-2}\yy_{\zeta,\nu})$
 \EndProcedure 
 \end{algorithmic}
\end{algorithm}

\begin{algorithm}
\caption{}
\label{alg:CKMST}
 \begin{algorithmic}[1]
 \Procedure{MWU}{$\AA, \PP, \xx,M,\zeta$}
 \State $\yy^{(0)}  = 0, \ww^{(0)} =\frac{\zeta}{m}$
 \State $ t = 0$
 \State $\AA' = \left[ \AA^{\top}, \PP^{\top}\nabla \ff(\xx)\right]^{\top}, \bb =[0,\frac{\zeta}{2}]$
\While{$\|\ww\|_1 \leq 10 \zeta$ }
	 \State $\Dtil \leftarrow {\arg \min}_{\AA'\Delta = \bb'} \sum_j \ff''(\xx_j)(\PP\Delta)_j^2 + 4 M^2 \sum_j \left(\ww^{(t)}_j + \frac{\|\ww^{(t)}\|_1}{m}\right)(\PP\Delta-\zz)_j^2$ 
	\If{$2M\norm{\PP\Dtil-\zz}_{\infty} \leq \tau $}\Comment{Flow Step}
 		\State $\yy^{(t+1)} \leftarrow \yy^{(t)} + \Dtil$
		\State $\ww^{(t+1)}\leftarrow \ww^{(t)}\left(1 + \frac{1}{2}\alpha M |\PP\Dtil-\zz|\right)$
	 \Else
		 \For{Indices $i$ such that $2M|\PP\Dtil-\zz|_i \geq \tau$}
		 	\State $\ww^{(t+1)}_i \leftarrow 2 \ww^{(t)}_i$\Comment{Width Step}
		\EndFor
	\EndIf
  \State $ t \leftarrow t+1$
\EndWhile
\State \Return $\frac{\yy^{(t)}}{100 t}$
 \EndProcedure 
 \end{algorithmic}
\end{algorithm}

We will show that Algorithm \ref{alg:CKMST} returns $\Delta$ such that $\|\PP\Delta-\zz\|_{\infty} \leq \frac{1}{2M}$ and $res(\Delta) \geq \frac{1}{400}\zeta$.

\begin{restatable}{lemma}{WidthReduction}\label{lem:res-approx}
Let $\zeta$ be such that $res(\Dopt) \in (\zeta/2,\zeta]$. Algorithm  \ref{alg:CKMST} returns $\yy$ such that $\AA\yy = 0$, $\|\PP\yy-\zz\|_{\infty} \leq \frac{1}{2M}$ and $res(\yy) \geq \frac{1}{400} res(\Dopt)$ in $O(m^{1/3})$ calls to a linear system solver.
\end{restatable}

We now state the main result of the section which follows directly from Lemmas \ref{lem:iterative-ref}, \ref{lem:res-opt} and \ref{lem:res-approx}.
\begin{theorem}\label{thm:main-eps}
For $\epsilon >0$, $M$-q.s.c. function $\ff$ and, $\xx^{(0)}$ such that $\AA\xx^{(0)} = \bb$, $\|\xx^{(0)}\|_{\infty} \leq R$, Algorithm \ref{alg:eps} finds $\xxtil$ such that $\AA\xxtil = \bb$ and $\ff(\xxtil) - \ff(\xx^{\star}) \leq \epsilon$ in $\Otil\left(MR m^{1/3}\log (MR) \log \left(\frac{\ff(\xx^{(0)})-\ff(\xx^{\star})}{\epsilon}\right)\right)$ calls to a linear system solver.
\end{theorem}



\section{Applications}\label{sec:SpecialFunc}
We now show how our methods may be applied to various quasi-self-concordant functions.
\subsection{Sum of Exponentials, Softmax and \texorpdfstring{$\ell_{\infty}$} --regression}
We recall the softmax function $\textrm{smax}_{\nu}(\PP\xx) = \nu \log \left(\sum_i e^{\frac{(\PP\xx)_i}{\nu}}\right)$, which we may note is $1/\nu$-q.s.c. We start by assuming that at the optimum, for $R \geq \Omega\left((\log m)^{-1}\right)$, $\textrm{smax}_{\nu}(\PP\xx^{\star}) \leq R$. 

We apply Algorithm \ref{alg:MainAlgo} to $\sum_i e^{\frac{(\PP\xx)_i}{\nu}}$, which is also $1/\nu$-q.s.c. We can use the following weight update step for the width reduction step: $\ww^{(t,k+1)}_i \leftarrow \ww^{(t,k)}_i + \nu \log (1+\epsilon)$.

\begin{restatable}{theorem}{smax}{\label{thm:smax}}
Let $\xx^{\star}$ denote the optimum of $\min_{\AA\xx=\bb}\emph{\textrm{smax}}_{\nu}(\PP\xx)$. Algorithm \ref{alg:MainAlgo}, when applied to the function $\ff(\PP\xx) = \sum_i e^{\frac{(\PP\xx)_i}{\nu}}$, returns $\xxtil$ such that $\AA\xxtil = \bb$, and 
\[
\emph{\textrm{smax}}_t(\PP\xxtil) \leq (1+\Otil(\nu))\emph{\textrm{smax}}_{\nu}(\PP\xx^{\star}),
\]
in at most $\Otil(m^{1/3}\nu^{-5/3})$ calls to a linear system solver. 
\end{restatable}
\begin{proof}
We know that $\textrm{smax}_{\nu}(\PP\xxtil) \leq \|\PP\xxtil\|_{\infty} + {\nu}\log m.$ From Lemma \ref{lem:runtime-qsc}, we have that $\xxtil$ is obtained in at most $\Otil(m^{1/3}\nu^{-5/3})$ calls to a linear system solver satisfying $\AA\xxtil = \bb$. Further, we also have, $\|\PP\xxtil\|_{\infty} \leq MR\|\ww^{(T,K)}\|_{\infty} = R \frac{\|\ww^{(T,K)}\|_{\infty}}{\nu}$. We will now bound $\frac{\|\ww^{(T,K)}\|_{\infty}}{\nu}$. We note that $\Phi(\ww^{(T,K)}) \leq \Phi(\ww_0) e^{1+4\nu}.$ For $\ww_0 = 0$,
\[
\Phi(\ww^{(T,K)}) = \frac{1}{\nu^2}\sum_i e^{\frac{\ww^{(T,K)}_i}{\nu}} = \Phi(\ww_0)\sum_i e^{\frac{\ww^{(T,K)}_i}{\nu}} \leq \Phi(\ww_0)e^{1+4\nu}.
\]
Therefore, we must have $\ww^{(T,K)} \leq (1+4\nu)\nu$. Our bound is
\[
\textrm{smax}_{\nu}(\PP\xxtil) \leq (1+4\nu)R + \nu\log m \leq \left(1+\Otil(\nu)\right) R,
\]
for $R\geq \Omega(1/\log m)$.
We can now do a binary search on $R$ as follows to obtain
\[
\textrm{smax}_{\nu}(\PP\xxtil) \leq (1+\Otil(\nu))\textrm{smax}_{\nu}(\PP\xx^{\star}).
\]
\paragraph{Binary search on $R$:}  Let $R_0$ denote the value $\|\PP\xx^{\star}\|_{\infty} = R_0$. Now, for any $R \geq R_0$, we attain an $\xxtil$ which has an objective value at most $R(1+4\nu)$. For any $R <R_0$, as long as $R$ is such that the plane $\AA\xx=\bb$ has at least one point with infinity norm at most $R$, we will get a feasible solution to our problem. However, the objective value guarantee of $R(1+4 \nu)$ may not hold. Since the optimum is $R_0$, the solution returned in such cases must give an objective value larger than $R_0$. We can thus do a binary search on $R$ and reach $O(\nu)$ close to the value $R_0$. This will require running our algorithm $O\left(\log(R_0 \nu^{-1})\right)$ times. In the end we can return the $\xx$ which gives the smallest objective values among all these runs.
\end{proof}

\begin{restatable}{theorem}{InftyReg}\label{thm:inftyreg}
 Let $\xx^{\star}$ denote the optimum of the $\ell_{\infty}$-regression problem, $\min_{\AA\xx=\bb}\|\PP\xx\|_{\infty}$. Algorithm \ref{alg:MainAlgo} when applied to the function $\ff(\PP\xx) = \sum_i \left(e^{\frac{(\PP\xx)_i}{\nu}}+e^{\frac{-(\PP\xx)_i}{\nu}}\right)$ for $\nu = \Omega\left(\frac{\epsilon}{\log m}\right)$, returns $\xxtil$ such that $\AA\xxtil = \bb$ and
 \[
 \|\PP\xxtil\|_{\infty} \leq \left(1+ \epsilon\right)\|\PP\xx^{\star}\|_{\infty},
 \]
 in at most $\Otil(m^{1/3}\epsilon^{-5/3})$ calls to a linear system solve.
\end{restatable}

\begin{restatable}{theorem}{ExpSumFinal}
For $\delta >0$, let $\xxbar$ be the solution returned by Algorithm \ref{alg:MainAlgo} (with $\epsilon = 1$) applied to $\ff(\PP\xx) = \sum_i e^{\frac{(\PP\xx)_i}{\nu}}$. Now, Algorithm \ref{alg:eps} with starting solution $\xx^{(0)} = \xxbar$, applied to $\ff$ finds $\xxtil$ such that $\AA\xxtil = \bb$ and $\sum_i e^{\frac{(\PP\xxtil)_i}{\nu}} \leq (1+\delta)\sum_i e^{\frac{(\PP\xx^{\star})_i}{\nu}}$ in at most $O\left(m^{1/3}R^2 \nu^{-2}\log \left(\frac{m}{\delta}\right)\right)$ calls to a linear system solver.
\end{restatable}

\subsection{$p$-Norm Regression}
We will solve, $\min_{\AA\xx= \bb} \ff(\PP\xx) = \|\PP\xx\|_p^p + \mu\|\PP\xx\|_2^2 $, for $p \geq 3$ which is $p\mu^{-1/(p-2)}$-q.s.c. w.r.t. its argument. We first apply Algorithm \ref{alg:MainAlgo} to this function and use the returned solution as the starting point of Algorithm \ref{alg:eps}. We can use the following weight update step for the width reduction step: $\ww^{(t,k+1)}_i \leftarrow (1+\epsilon)^{1/(p-2)}\ww^{(t,k)}_i$.

\begin{restatable}{theorem}{lpreg}\label{thm:lpreg}
For $\delta >0$ and $p\geq 3$, let $\xxbar$ be the solution returned by Algorithm \ref{alg:MainAlgo} (with $\epsilon = 1$) applied to $\ff(\PP\xx) = \|\PP\xx\|_p^p + \mu\|\PP\xx\|_2^2$. Now, Algorithm \ref{alg:eps} with starting solution $\xx^{(0)} = \xxbar$, applied to $\ff$ finds $\xxtil$ such that $\AA\xxtil = \bb$ and $\ff(\PP\xxtil) \leq \ff(\PP\xx^{\star}) + \delta$ in at most $O\left(p^2 \mu^{-1/(p-2)}m^{1/3}R\log \left(\frac{p m R}{\mu \delta }\right)\right)$ calls to a linear system solver.
\end{restatable}

\subsection{Logistic Regression}

We consider the function $\ff(\PP\xx) = \sum_i \log(1 + e^{(\PP\xx)_i})$ which is $1$-q.s.c. w.r.t its argument. We will use Algorithm \ref{alg:MainAlgo} with the following weight update for the width reduction step which reduces the resistance by a factor of $(1+\epsilon)$: $\ww^{(t,k+1)}_i \leftarrow \ww^{(t,k)}_i + 0.9\epsilon$

\begin{restatable}{theorem}{logisticRuntime}\label{thm:logisticRuntime}
For $\delta >0$, let $\xxbar$ be the solution returned by Algorithm \ref{alg:MainAlgo} (with $\epsilon = 1$) applied to $\ff(\PP\xx) = \sum_i \log(1 + e^{(\PP\xx)_i})$. Now, Algorithm \ref{alg:eps} with starting solution $\xx^{(0)} = \xxbar$, applied to $\ff$ finds $\xxtil$ such that $\AA\xxtil = \bb$ and $\sum_i \log(1 + e^{(\PP\xxtil)_i}) \leq \sum_i \log(1 + e^{(\PP\xx^{\star})_i}) +\delta$ in at most $O\left(m^{1/3}R\log \left(\frac{m R}{\delta}\right)\right)$ calls to a linear system solver.
\end{restatable}

\newpage

\bibliographystyle{alpha}
\bibliography{papers}
\newpage

\newpage
\appendix


\section{Algorithm for General-Self-Concordant Functions}\label{app:gsc}

In this section we will show how to use our algorithms for the following classes of general-self-concordant functions.
\begin{enumerate}
\item $6> \nu \geq 2$: $\ff$ is $(N,\nu)$-g.s.c. and $L$-smooth.
\item $\nu<2$: $\ff$ is $(N,\nu)$-g.s.c., $L$-smooth and $\mu$-strongly convex.
\end{enumerate}

We will use the following result to reduce these problems to $(M,2)$-g.s.c. problems and use our algorithms.
\begin{lemma}[Prop 4. \cite{SunT19}] \label{lem:smooth}
Let $\ff$ be $(M,\nu)$-g.s.c. with $\nu>0$. Then:
\begin{enumerate}[(a)]
\item If $\nu \in (0,3]$ and $\ff$ is also strongly convex with strong convexity parameter $\mu>0$ in $\ell_2$-norm, then $\ff$ is also $\left(\frac{M}{\sqrt{\mu}^{3-\nu}},3\right)$-g.s.c.
\item If $\nu \geq 2$ and $\nabla \ff$ is Lipschitz continuous with finite Lipschitz constant $L$ in $\ell_2$-norm, then $\ff$ is also $\left(ML^{\frac{\nu}{2}-1},2\right)$-g.s.c.
\end{enumerate}
\end{lemma}

We thus have the following result.

\begin{theorem}
For $\delta>0$, $\ff$ $(N,\nu)$-g.s.c. $6> \nu\geq 2$ and $L$-smooth, let $\xxbar$ be the solution returned by Algorithm \ref{alg:MainAlgo} (with $\epsilon = 1$) applied to $\ff(\xx)$. Now, Algorithm \ref{alg:eps} with starting solution $\xx^{(0)} = \xxbar$, applied to $\ff$ finds $\xxtil$ such that $\AA\xxtil = \bb$ and $\sum_i \ff(\PP\xxtil_i) \leq \sum_i\ff(\PP\xx^{\star}_i)+\delta$ in at most 
\[O\left(m^{1/3}NL^{\frac{\nu-2}{2}}R\log \left(\frac{\ff(\xx^{(0)})-\ff(\xx^{\star}))}{\delta}\right)\right)
\]
calls to a linear system solver.

\end{theorem}
\begin{proof} From Lemma \ref{lem:smooth}, $\ff$ is $(NL^{(\nu-2)/2},2)$-g.s.c. We now use Lemma \ref{lem:runtime-qsc} with $M = NL^{(\nu-2)/2}$ followed by Theorem \ref{thm:main-eps}. 
\end{proof} 

\begin{theorem} 
For $\delta >0$, $\ff$ $(N,\nu)$-g.s.c. $2> \nu\geq 0$ and $L$-smooth $\mu$-strongly convex, let $\xxbar$ be the solution returned by Algorithm \ref{alg:MainAlgo} (with $\epsilon = 1$) applied to $\ff(\xx)$. Now, Algorithm \ref{alg:eps} with starting solution $\xx^{(0)} = \xxbar$, applied to $\ff$ finds $\xxtil$ such that $\AA\xxtil = \bb$ and $\sum_i \ff(\PP\xxtil_i) \leq \sum_i\ff(\PP\xx^{\star}_i)+\delta$ in at most 
\[O\left(m^{1/3}N\mu^{-\frac{3-\nu}{2}}L^{1/2}R\log \left(\frac{\ff(\xx^{(0)})-\ff(\xx^{\star}))}{\delta}\right)\right)
\]
calls to a linear system solver.

\end{theorem}
\begin{proof}
From Lemma \ref{lem:smooth}, $\ff$ is $(N\mu^{-\frac{3-\nu}{2}}L^{1/2},2)$-g.s.c. We now use Lemma \ref{lem:runtime-qsc}with $M = N\mu^{-\frac{3-\nu}{2}}L^{1/2}$ followed by Theorem \ref{thm:main-eps}.
\end{proof}


\section{Missing Proofs}

\subsection{Proofs from Section \ref{sec:prelims}}

\begin{restatable}{definition}{HessianStab}[Hessian Stability]
For distance $r \in \rea_{\geq 0}$ and function $\dd:\rea_{\geq 0} \rightarrow \rea_{\geq 0}$ acting on $r$, a function $\ff$ is $(r,\dd(r))$-hessian stable w.r.t. a norm $\|\cdot\|$ if for all $\xx, \yy$ such that $\|\xx-\yy\|\leq r$, 
\[
\frac{1}{\dd(r)} \nabla^2 \ff(\xx) \preceq \nabla^2 \ff(\yy) \preceq \dd(r) \nabla^2 \ff(\xx)
\]
\end{restatable}

\begin{lemma}[Lemma 11 \cite{CJJJLST}]\label{fact:qsc}
 If $\ff$ is a univariate $M$-quasi-self-concordant (q.s.c.) function, then $\ff(\xx) = \sum_i \ff(\xx_i)$ is $(r, e^{Mr})$ hessian stable in the $\ell_{\infty}$-norm.
 \end{lemma}

\PhiPsi*
\begin{proof}
Let $\Dtil$ be the minimizer of $\Psi(\rr)$ and $\xx^{\star}$ be the optimum of \eqref{eq:Problem}.
\begin{align*}
\Psi(\rr) & =  \sum_i \rr_i (\PP\Dtil)_i^2 \leq  \sum_i \rr_i (\PP\xx^{\star})_i^2 \\
& =  \sum_i \left(\ff''(\ww_i) + \frac{\epsilon\Phi(\ww)}{m} \right)\frac{(\PP\xx^{\star})_i^2}{R^2}\\
& \leq \sum_i \ff''(\ww_i)+ \frac{\epsilon \Phi(\ww)}{m}\cdot m, && \text{Since $\|\PP\xx^{\star}\|_{\infty}\leq R$}\\
& = \Phi(\ww) (1+\epsilon)
\end{align*}
We next look at a lower bound for $\Psi$. We note that, any solution to the oracle must satisfy $\AA\Dtil = \bb$. This implies, $\|\AA\| \|\Dtil\|_2 \geq \|\bb\|_2$, where $\| \cdot \|$ denotes the operator norm. Now,
\[
\Psi(\rr) \geq \frac{\epsilon \Phi(\ww)}{m R^2} \|\PP\Dtil\|_2^2 \geq \frac{\epsilon\Phi(\ww)}{m R^2}\|\PP\|_{\min}^2\|\Dtil\|_2^2\geq \frac{\epsilon\Phi(\ww)}{m R^2}  \frac{\|\PP\|_{\min}^2\|\bb\|_2^2}{\|\AA\|^2}.
\]
\end{proof}

\begin{lemma}\label{lem:BoundLinear}
\[
\sum_i \ff''(\ww_i) |\PP\Dtil|_i \leq (1+\epsilon) R \Phi(\ww)
\]
\end{lemma}
\begin{proof}
\begin{align*}
\sum_i \ff''(\ww_i) |\PP\Dtil|_i & \leq \sqrt{\sum_i \ff''(\ww_i) \sum_i\ff''(\ww_i) |\PP\Dtil|_i^2} && \text{Cauchy Schwarz}\\
&\leq \sqrt{\Phi(\ww)}\sqrt{R^2 \Psi(\rr)}\\
& \leq R \sqrt{\Phi(\ww)}\sqrt{(1+\epsilon)\Phi(\ww)} &&\text{From Lemma \ref{lem:Relate-Psi-Phi}}\\
& = R(1+\epsilon) \Phi(\ww).
\end{align*}
\end{proof}

\subsection{Proofs from Section \ref{sec:CrudeOracle}}

\subsubsection*{Change in $\Psi$}
\Width*
\begin{proof}
We show this by induction. It is clear that this holds for $t=k=0$. We know from Lemma \ref{lem:ckmst:res-increase}, for $\rr' \geq \rr$, 
\[
\Psi(\rr') \geq \Psi(\rr)+ \sum_i \left(1-\frac{\rr_i}{\rr_i'}\right) \rr_i (\PP\Dtil)_i^2.
\]
Since the weights are only increasing, this corresponds to the case $\ff''$ is an increasing function. Similarly, when $\ff''$ is a non-increasing function, we have the following bound: for $\rr' \leq \rr$ from Lemma \ref{lem:r-dec},
\[
\Psi(\rr') \leq \Psi(\rr) -  \frac{1}{2}\sum_i \left(1-\frac{\rr'_i}{\rr_i}\right) \rr_i (\PP\Dtil)_i^2.
\]
We first consider a flow step. We note that our weights $\ww$ are increasing, and if $\ff''$ is increasing then $\rr^{(t+1)} \geq \rr^{(t)}$. Similarly if $\ff''$ is decreasing, $\rr^{(t+1,k)} \leq \rr^{(t,k)}$.  We can use the above relations to now get $\Psi(\rr^{(t+1,k)}) \geq \Psi(\rr^{(t,k)})$ for the first case and $\Psi(\rr^{(t+1,k)}) \leq \Psi(\rr^{(t,k)})$ for the second. We next consider a width reduction step. Let $i$ be one edge that has $|\PP\Dtil_i| \geq R \tau$. We have,
\[
\rr^{(t,k)}_i (\PP\Dtil)_i^2 \geq \frac{\epsilon \Phi(\ww^{(t,k)})}{R^2 m} |\PP\Dtil|_i^2 \geq \frac{\epsilon \Phi(\ww^{(t,k)})}{R^2 m} R^2 \tau^2 \geq \frac{\epsilon \tau^2}{(1+\epsilon)m} \Psi(\rr^{(t,k)}),
\]
where the last inequality follows from Lemma \ref{lem:Relate-Psi-Phi}. Now, since we are changing our resistances by a factor of $(1+\epsilon)$, we get the following bounds for the two cases,
\[
\Psi(\rr^{(t,k+1)}) \geq \Psi(\rr^{(t,k)})+  \left(1-\frac{\rr_i}{(1+\epsilon)\rr_i}\right)\frac{\epsilon\tau^2}{(1+\epsilon)m}\Psi(\rr^{(t,k)}) =\Psi(\rr^{(t,k)})\left(1 +  \frac{\epsilon^2 \tau^2}{(1+\epsilon)^2m}\right),
\]
\[
\Psi(\rr^{(t,k+1)}) \leq \Psi(\rr^{(t,k)}) -  \frac{1}{2} \left(1-\frac{\rr_i/(1+\epsilon)}{\rr_i}\right)\frac{\epsilon\tau^2}{(1+\epsilon)m} \Psi(\rr^{(t,k)})= \Psi(\rr^{(t,k)})\left(1 - \frac{\epsilon^2\tau^2}{2(1+\epsilon)^2m} \right).
\]
With these two relations we conclude our proof.
\end{proof}

\subsubsection*{Change in $\Phi$}

\Flow*
\begin{proof}
We first show the case when $\ff''$ is increasing. The same calculation will work for the other case too by just considering the sign of $\Phi'$. 

We will use induction. It is easy to see the claim holds for the initial iteration, $t=k=0$. We next assume that it holds for some $\ww^{(t,k)}$. If the next step is a flow step, we update to $\ww^{(t+1,k)} \leq \ww^{(t,k)}+ \epsilon \alpha  \tau$.  Since $\alpha \tau \leq M^{-1}$, we have that $\Phi$ is $(M^{-1},e^{\epsilon})$ hessian stable around this update. We will use $\ww$ to denote $\ww^{(t,k)}$ for simplicity. We thus have,
\begin{align*}
\Phi(\ww^{(t+1)}) =&\Phi\left(\ww + \frac{\epsilon\alpha}{R} |\PP\Dtil|\right)\\
= & \Phi(\ww) +\frac{\epsilon\alpha}{R} \nabla \Phi(\yy)^{\top}|\PP\Dtil| \\
&\text{(For some $\yy$ between $\ww$ and $\ww+\alpha|\PP\Delta|$)}\\
= &  \Phi(\ww) +\frac{\epsilon\alpha}{R} \sum_i \ff'''(\yy_i)|\PP\Dtil|_i  \\
\leq &  \Phi(\ww) +\frac{\epsilon\alpha}{R} M \sum_i \ff''(\yy_i)|\PP\Dtil|_i\\
& \text{(Since $\ff$ is $M$-q.s.c.)}\\
\leq & \Phi(\ww) + \frac{\epsilon\alpha}{R} M e^{\epsilon} \sum_i \ff''(\ww_i)|\PP\Dtil|_i\\
&\text{(Since $\ff$ is hessian stable in this range)}\\
\leq  & \Phi(\ww) +  \epsilon(1+\epsilon)^2\alpha M \Phi(\ww) \\
&\text{(From Lemma \ref{lem:BoundLinear})}
\end{align*}
We thus get the following bound,
\[
\Phi(\ww^{(t+1,k)}) \leq \Phi(\ww^{(t,k)})\left( 1 + \epsilon(1+\epsilon)^2 \alpha M \right).
\]

Now, suppose the next step is a width reduction step. 

\begin{align*}
\Phi(\ww^{(t,k+1)}) =& \sum_{i \notin \mathcal{I}}\Phi(\ww_i) + \sum_{i \in \mathcal{I}}\Phi\left(\ww^{(t+1)}_i\right)\\
= & \sum_{i \notin \mathcal{I}}\Phi(\ww_i) + \sum_{i \in \mathcal{I}}\ff''\left(\ww^{(t+1)}_i\right) \\
 \leq & \sum_{i \notin \mathcal{I}}\Phi(\ww_i) + (1+\epsilon) \sum_{i \in \mathcal{I}}\ff''\left(\ww_i\right) \\
\leq & \Phi(\ww) + \frac{\epsilon}{R\tau}\sum_{i \in \mathcal{I}} \ff''(\ww_i)|\PP\Dtil|_i \\
 \leq & \Phi(\ww) + \frac{\epsilon}{R\tau}\sum_{i} \ff''(\ww_i)|\PP\Dtil|_i \\
 \leq & \Phi(\ww) + \frac{\epsilon(1+\epsilon)}{\tau}\Phi(\ww) \\
 &\text{From Lemma \ref{lem:BoundLinear}}
\end{align*}
We thus get the following bound,
\[
\Phi(\ww^{(t,k+1)}) \leq \Phi(\ww^{(t,k)})\left( 1 + \epsilon (1+\epsilon) \tau^{-1} \right).
\]
\end{proof}

\subsection{Proofs from Section \ref{sec:boost}}

\subsubsection*{Iterative Refinement}
\begin{lemma}\label{lem:bound}
Let $\ff$ be a $(r,d(r))$-hessian stable function in $\ell_{\infty}$-norm, and $\xxtil = \xx + \Delta$ such that $\|\Delta\|_{\infty} \leq r$. We then have,
\[
\frac{1}{d(r)} \Delta^{\top}\nabla^2 \ff(\xx) \Delta \leq \ff(\xxtil) - \ff(\xx) -  \nabla \ff(\xx)^{\top} \Delta \leq d(r) \Delta^{\top}\nabla^2 \ff(\xx) \Delta,
\]
\end{lemma}

\begin{proof}
We have for some $\zz$ along the line joining $\xx$ and $\xxtil$,
\[
 \ff(\xxtil) = \ff(\xx) + \nabla \ff(\xx)^{\top} \Delta +  \Delta^{\top}\nabla^2 \ff(\zz) \Delta.
\]
Since $\|\zz-\xx\|_{\infty} \leq \|\xxtil-\xx\|_{\infty} \leq r$, from hessian stability, we have,
\[
\frac{1}{d(r)} \nabla^2 \ff(\xx) \preceq \nabla^2 \ff(\zz) \preceq d(r) \nabla^2 \ff(\xx).
\]
Using this relation in the above, we get our lemma.
\end{proof}

 \begin{lemma}\label{lem:residual}
   Let $\Delta$ be any feasible solution to the residual problem at
   $\xx$. We then have,
\[
\ff(\xx) - \ff(\xx-\Delta) \leq res(\Delta), \quad \ff(\xx) - \ff(\xx - e^{-2}\Delta) \geq e^{-2} \cdot res(\Delta),
\]
\end{lemma}
\begin{proof}
Since our function is $M$-q.s.c., from Lemmas \ref{lem:bound} and \ref{fact:qsc}, for all $\Delta$ such that $\|\PP\Delta\|_{\infty} \leq M^{-1}$,
\[
 e^{-1} (\PP\Delta)^{\top}\nabla^2 \ff(\xx) \PP\Delta \leq \ff(\xx-\Delta) - \ff(\xx)  +\nabla \ff(\xx)^{\top} \PP\Delta  \leq  e (\PP\Delta)^{\top}\nabla^2 \ff(\xx) \PP\Delta.
\]
The first bound directly follows from the left inequality. For the second bound, we first note that $e^{-2}\|\PP\Delta\| \leq M^{-1}$. We can now use the right inequality. 
\begin{align*}
\ff(\xx) - \ff(\xx- e^{-2}\Delta) & \geq e^{-2}\nabla \ff(\xx)^{\top} \PP\Delta  - e^{-3} (\PP\Delta)^{\top}\nabla^2 \ff(\xx) \PP\Delta\\
& = e^{-2} \left(\nabla \ff(\xx)^{\top} \PP\Delta  - e^{-1} (\PP\Delta)^{\top}\nabla^2 \ff(\xx) \PP\Delta\right)\\
& = e^{-2} res(\Delta).
\end{align*}
\end{proof}

\begin{lemma}\label{lem:res-lower-bound}
Assume $\ff$ is $M$-q.s.c. Let $\xx^{\star}$ denote the minimizer of Problem \eqref{eq:Problem} and $\Dopt$ the optimizer of Problem \eqref{eq:res-prob-red} at $\xx^{(t)}$. We then have,
\begin{equation*}
 res(\Dopt) \geq \frac{1}{4MR}\left(\ff(\xx^{(t)}) - \ff(\xx^{\star})\right).
\end{equation*}
\end{lemma}

\begin{proof}
  Let $\xx^{(t)}$ be such that $\AA\xx^{(t)} = \bb$ and
  $\xx^{\star}$ is the optimum of \eqref{eq:Problem}. Note that we have $\|\PP\xx^{(t)}\|_{\infty}\leq R$ and therefore,   $\norm{\PP\xx^{(t)} - \PP\xx^{\star}}_{\infty} \leq 2R $. Let
  $r = \frac{1}{2M}$ and
  $\xx = \left(1-\frac{r}{2R}\right)\xx^{(t)} + \frac{r}{2R}
  \xx^{\star}$. Let $\Dtil = \xx^{(t)} - \xx = \frac{r}{2R} (\xx^{(t)}- \xx^{\star})$. We have,
\[
 \norm{\PP\Dtil}_{\infty} =  \norm{\PP\xx^{(t)}- \PP\xx}_{\infty} =
  \frac{r}{2R}\norm{\PP\xx^{(t)}-\PP\xx^{\star}}_{\infty} \leq r,
\]
and 
\[
\AA\Dtil = \AA(\xx^{(t)}-\xx) = \frac{r}{2R}(-\AA\xx^{\star} + \AA\xx^{(t)} ) = 0.
\]
We next show that $\norm{\PP\Dtil - \zz}_{\infty} \leq \frac{1}{2M}$.
\begin{align*}
\norm{\PP\Dtil - \zz}_{\infty} & = \norm{\frac{r}{2R}\PP\xx^{(t)} - \frac{r}{2R}\PP\xx^{\star} - \zz}_{\infty}               
\end{align*}
We will do a case by case analysis. Consider some coordinate $i$.
\begin{enumerate}
\item $\PP\xx^{(t)}_i -\frac{1}{2M} < - R$: From the definition of $\zz_i$, we note that $\zz_i = R - \frac{1}{2M} + \PP\xx^{(t)}_i$ and $-R < \PP\xx^{(t)}_i \leq - R + \frac{1}{2M} $. Suppose $\PP\xx^{(t)}_i = - R+a$ for some $0\leq a <\frac{1}{2M}$. We have,
\begin{align*}
\abs{\PP\Dtil - \zz}_i  = &\abs{\frac{r}{2R}(\PP\xx^{(t)}_i - \PP\xx^{\star}_i) - \zz_i}\\
 = &\abs{\frac{r}{2R}(- R + a - \PP\xx^{\star}_i) - a + \frac{1}{2M}  } \\
= & \abs{\frac{r}{2R}(- R - \PP\xx^{\star}_i) - a \left(1-\frac{r}{2R}\right) + \frac{1}{2M}  }   \\
\leq & \frac{1}{2M}.
\end{align*}
The last inequality follows since $-2R \leq -R- \PP\xx^{\star}_i \leq 0$.
 \item $\PP\xx^{(t)}_i + \frac{1}{2M} >R$: From the definition of $\zz_i$, we note that $\zz_i = -R + \frac{1}{2M} + \PP\xx^{(t)}_i$ and $R - \frac{1}{2M} < \PP\xx^{(t)}_i \leq R$. Suppose $\PP\xx^{(t)}_i = R-a$ for some $0\leq a <\frac{1}{2M}$. We have,
\begin{align*}
\abs{\PP\Dtil - \zz}_i  = &\abs{\frac{r}{2R}(\PP\xx^{(t)}_i - \PP\xx^{\star}_i) - \zz_i}\\
 = &\abs{\frac{r}{2R}(R-a - \PP\xx^{\star}_i) + a - \frac{1}{2M}  } \\
= & \abs{\frac{r}{2R}(R- \PP\xx^{\star}_i) + a \left(1-\frac{r}{2R}\right) - \frac{1}{2M}  }   \\
\leq & \frac{1}{2M}.
\end{align*}
The last inequality follows since $0 \leq R- \PP\xx^{\star}_i \leq 2R$.
\item $-R +\frac{1}{2M}  \leq \PP\xx^{(t)}_i  \leq  - \frac{1}{2M}R$: In this case $\zz_i = 0$.
\begin{align*}
\abs{\PP\Dtil - \zz}_i  = &\abs{\frac{r}{2R}(\PP\xx^{(t)}_i - \PP\xx^{\star}_i)} \leq r = \frac{1}{2M}.
\end{align*}
\end{enumerate}
We thus conclude, that $\xx-\xx^{(t)}$ is a feasible solution for the residual problem and from convexity,
\[
 \frac{r}{2R} \left(\ff(\xx^{(t)}) - \ff(\xx^{\star})\right) \leq \ff(\xx^{(t)}) -\ff(\xx).
\]
Let $\Dopt$ denote the optimum of the residual problem at $\xx^{(t)}$ \eqref{eq:res-prob-red}. From Lemma \ref{lem:residual},
\begin{equation*}\label{eq:res-lower-bound}
  \frac{r}{2R} \left(\ff(\xx^{(t)}) - \ff(\xx^{\star})\right) \leq \ff(\xx^{(t)}) -\ff(\xx) \leq  res\left(\xx^{(t)}-\xx\right) \leq res(\Dopt).
\end{equation*}
\end{proof}

\Iterative*
\begin{proof}
  From Lemma \ref{lem:res-lower-bound},
\[
  res(\Dtil^{(t)}) \geq \frac{1}{\kappa}res(\Dopt) \geq
  \frac{1}{4\kappa MR} \left(\ff(\xx^{(t)}) -
    \ff(\xx^{\star})\right).
\]
Now, from Lemma \ref{lem:residual},
\[
\ff(\xx^{(t+1)}) - \ff(\xx^{\star}) \leq \ff(\xx^{(t)}) -  \ff(\xx^{\star}) - e^{-2} res(\Dtil^{(t)}) \leq \left(1 - \frac{e^{-2}}{4 \kappa MR}\right) \left(\ff(\xx^{(t)}) -  \ff(\xx^{\star})\right).
\]
Inductively applying the above equation,
\[
\ff(\xx^{(T)}) - \ff(\xx^{\star}) \leq \left(1 - \frac{e^{-2}}{4 \kappa MR}\right)^T \left(\ff(\xx^{(0)}) -  \ff(\xx^{\star})\right).
\]
\end{proof}

\subsubsection*{Binary Search}

\binaryNu*
\begin{proof}
The lower bound follows form \ref{lem:res-lower-bound}. For the upper bound, from \ref{lem:residual},
\[
\nu \geq \ff(\xx^{(t)}) - \ff(\xx^{\star}) \geq \ff(\xx^{(t)}) - \ff(\xx-e^{-2}\Dopt) \geq e^{-2} res(\Dopt).
\]
\end{proof}

\binaryZeta*
\begin{proof}
Consider scaling $\Dopt$ by $ O(1)>\lambda > 0$. We must have,
\[
\left[\frac{d}{d\lambda}res(\lambda \Dopt)\right]_{\lambda = 1} = 0.
\]
This implies,  
\[
\nabla \ff(\xx)^{\top}\PP\Dopt - 2e^{-1} (\PP\Dopt)^{\top} \nabla^2 \ff(\xx) \PP\Dopt = 0,
\]
or
\[
e^{-1}(\PP\Dopt)^{\top} \nabla^2 \ff(\xx) \PP\Dopt = \nabla \ff(\xx)^{\top}\PP\Dopt - e^{-1} (\PP\Dopt)^{\top} \nabla^2 \ff(\xx) \PP\Dopt = res(\Dopt) \leq \zeta.
\]
\end{proof}

\subsubsection*{Width Reduction}
\WidthReduction*
\begin{proof}
This algorithm is basically an implementation of the width-reduced MWU algorithm from \cite{CKMST}. We will give a proof for completeness. For the purpose of this proof, we denote,
\[
\Psi(\rr) = \min_{\AA\Delta = 0, \nabla \ff(\xx)^{\top}\PP\Delta = \zeta/2}  \quad   \sum_j \left(\ff''(\xx_j)(\PP\Delta)_j^2 + \sum_j 4M^2 \left(\ww_j + \frac{\|\ww\|_1}{m}\right)\right)(\PP\Delta-\zz)_j^2,
\]
\[
\Phi(\ww) =  \|\ww\|_1.
\]
Let $\Dtil$ be the solution returned by $\Psi$. We first note that, for $\Dopt$ the optimum of the residual problem,
\begin{align*}
\Psi(\rr) &\leq  \sum_j \left(\ff''(\xx_j)(\PP\Dopt)_j^2 + \sum_j 4M^2 \left(\ww_j + \frac{\|\ww\|_1}{m}\right)\right)(\PP\Dopt-\zz)_j^2\\
& \leq e \cdot \zeta  + \sum_j  4M^2 \left(\ww_j + \frac{\|\ww\|_1}{m}\right)(\PP\Dopt-\zz)_j^2, \text{ From Lemma \ref{lem:binary}}\\
& \leq e \cdot \zeta + \|\ww\|_1 + \Phi(\ww), \text{ Since $\|\PP\Dopt-\zz\|_{\infty} \leq \frac{1}{2M} $}\\
& \leq (e+2) \Phi(\ww).
\end{align*}
We note that,
\begin{equation}
\sum_j \ww_j (4M) (\PP\Dtil-\zz)_j \leq \sqrt{\sum_j \ww_j \sum_j \ww_j (4M)^2(\PP\Dtil-\zz)_j^2} \leq \sqrt{\Phi(\ww) \Psi(\rr)} \leq \sqrt{e+2} \Phi(\ww).
\end{equation}
For a flow step, from the above calculation, note that, 
\[
\Phi(\ww^{(t+1)}) = \sum_j \ww_j + \frac{\alpha}{2} \sum_j \ww_j M(\PP\Dtil-\zz)_j \leq \Phi(\ww^{(t)}) + \frac{\sqrt{e+2}}{8}\alpha \Phi(\ww^{(t)}) = \Phi(\ww^{(t)}) \left(1 + \alpha \right).
\]
For a width reduction step let $\mathcal{I}$ denote the indices which have the weights doubled,
\begin{align*}
\Phi(\ww^{(t+1)}) &= \sum_{j \notin \mathcal{I}} \ww^{(t)}_j + 2\sum_{j \in \mathcal{I}} \ww^{(t)}_j \leq \Phi(\ww^{(t)}) + \frac{2}{\tau} \sum_{j \in \mathcal{I}} \ww^{(t)}_j (2M) |\PP\Dtil-\zz|_j\\
& \leq \Phi(\ww^{(t)}) + \frac{\sqrt{e+2}}{\tau} \Phi(\ww) \leq \Phi(\ww^{(t)} \left(1 + 3 \tau^{-1}\right).
\end{align*}
We can bound the number of width reduction steps by $O(m/\tau^2)$ similar to Lemma \ref{lem:Width}. We now show that our final solution has $\|\frac{1}{T}\PP\yy-\zz\|_{\infty} \leq \frac{1}{2M}$. After $T$ iterations, let $j$ denote the index with max value in vector $\ww$. For $\alpha \tau \leq 1$, $\left(1 + \frac{\alpha}{2} M|\PP\Dtil-\zz|_j\right) \geq \exp\left(\frac{3}{4} \alpha M |\PP\Dtil-\zz|_j\right)$.
\begin{align*}
10 \zeta &\geq \Phi(\ww^{T}) \geq \ww_j^{(T)} \geq \frac{\zeta}{m}\Pi_{t = 1}^T\left(1 + \frac{\alpha}{2} M |\PP\Dtil^{(t)}-\zz|_j\right)\\
& \geq \frac{\zeta}{m} \exp \left(\frac{3}{8} \alpha (2M) \sum_t |\PP\Dtil^{(t)}-\zz|_j \right) = \frac{\zeta}{m} \exp \left(\frac{3}{8} \alpha (2M) (\PP\yy-T\zz)_j\right).
\end{align*}
We thus have for all coordinates $j$ and $T \geq \alpha^{-1} O(\log m)$,
\[
\frac{\abs{\PP\yy-T\zz}_j}{T } \leq \frac{O(M^{-1} \log m)}{\alpha T} \leq \frac{1}{2M}.
\]
It remains to show that $\yy/(100 T)$ has the required value for the residual. First note that, 
\[
\nabla \ff(\xx)^{\top}\frac{\yy}{100 T} = \frac{1}{100 T}\sum_t \nabla \ff(\xx)^{\top}\PP\Dtil^{(t)} = \frac{\zeta}{2\cdot  100 }.
\]
We next look at the quadratic term.
\begin{multline*}
\frac{1}{(100)^2T^2  }\sum_j \ff''(\xx_j) \yy_j^2 = \frac{1}{T^2(100)^2}  \sum_j \ff''(\xx_j)\left( \sum_t|\PP\Dtil^{(t)}|_j\right)^2 \\ \leq \frac{1}{T^2(100)^2}  \sum_j T \sum_t \ff''(\xx_j)|\PP\Dtil^{(t)}|_j^2  = \frac{1}{T(100)^2}  \sum_t  \Psi(\rr^{(t)}) \\ \leq \frac{1}{T(100)^2 } T (e+2) \Phi(\ww^{(T)}) \leq \frac{10(e+2)}{(100)^2} \zeta. 
\end{multline*}
Choose $c$ such that we have,
\[
e^{-1} \frac{1}{(100)^2}\sum_j \ff''(\xx_j) \yy_j^2 \leq \frac{\zeta}{4\cdot 100}.
\]
We thus have,
\[
res\left(\frac{\yy}{100 T}\right) = \nabla \ff(\xx)^{\top}\frac{\yy}{100 T} - e^{-1} \frac{1}{(100)^2T^2 }\sum_j \ff''(\xx_j) \yy_j^2 \geq \frac{\zeta}{4\cdot 100} \geq \frac{1}{400} res(\Dopt).
\]
\end{proof}

\subsection{Proofs from Section \ref{sec:SpecialFunc}}

\subsubsection*{Sum of exponential, soft-max and \texorpdfstring{$\ell_{\infty}$} -regression}

\InftyReg*
\begin{proof}
Let $\QQ = \begin{bmatrix}
           P \\
           -P 
         \end{bmatrix}$.
We note that $\ff(\xx) = \sum_i e^{\frac{(\QQ\xx)_i}{\nu}}$. Let $\xxbar$ denote the optimum of $\ff$, which is also the optimum of $smax_{\nu}(\QQ\xx)$. We have the following relation,
\[
\forall \xx, \|\PP\xx\|_{\infty} \leq smax_{\nu}(\QQ\xx) \leq \|\PP\xx\|_{\infty} + \nu\log m.
\]
Let $R = \|\PP\xx^{\star}\|_{\infty}$ (we can find this up to $\epsilon$ error using binary search), then the above relation implies $smax_{\nu}(\QQ\xxbar) \leq R(1+\epsilon).$ From Theorem \ref{thm:smax},
\[
\|\PP\xxtil\|_{\infty} \leq smax_{\nu}(\QQ\xxtil) \leq R\left(1+\epsilon\right) = \|\PP\xx^{\star}\|_{\infty}\left(1+\epsilon\right).\qedhere
\]
\end{proof}

\ExpSumFinal*

\begin{proof}
From Lemma \ref{lem:runtime-qsc}, Algorithm \ref{alg:MainAlgo} returns $\xxbar$ in $O(m^{1/3})$ iterations such that $\AA\xxbar = \bb$ and $\|\PP\xxbar\|_{\infty} \leq MR\|\ww^{(T,K)}\|_{\infty}$. Since $\frac{1}{\nu^2}\sum_i e^{\frac{\ww^{(T,K)}_i}{\nu}} = \Phi(\ww^{(T,K)}) \leq \Phi(\ww_0)e^{5}$, we have $\|\ww^{(T,K)}\|_{\infty} \leq 5\nu.$ This gives, $\|\PP\xxbar\|_{\infty} \leq 5R$. We next bound the function value.
\[
\ff(\PP\xxbar) = \sum_i e^{\frac{\PP\xx_i}{\nu}} \leq \sum_i e^{\frac{\ww^{(T,K)}_i MR}{\nu}}.
\]
If $MR\leq 1$, then $\ff(\PP\xxbar) \leq \nu^2 \Phi(\ww^{(T,K)}) \leq m$. Otherwise,
\[
\ff(\PP\xxbar) \leq \sum_i \left(e^{\frac{\ww^{(T,K)}_i}{\nu}}\right)^{MR} \leq \left(\sum_i e^{\frac{\ww^{(T,K)}_i}{\nu}}\right)^{MR} \leq (\nu^2 \Phi(\ww^{(T,K)}))^{MR} \leq O(m^{MR}).
\]
Now, we use Algorithm \ref{alg:eps}. Using the above calculated bounds in Theorem \ref{thm:main-eps}, we get our result.
\end{proof}

\subsubsection*{$\ell_p$-Regression}

\lpreg*
\begin{proof}
From Lemma \ref{lem:runtime-qsc}, we get $\xxbar$ such that $\|\xxbar\|_{\infty} \leq RM\|\ww^{(T,K)}\|_{\infty}$. We now want to bound $\ff(\xxbar)$.
\[
\ff(\xxbar) = (RM)^p \|\ww^{(T,K)}\|_p^p +\mu (RM)^2 \|\ww^{(T,K)}\|_2^2.
\]
We next note that for $\ww^{(T,K)}\geq \ww_0 =1$,
\[
\Phi(\ww^{(T,K)}) = p(p-1)\|\ww^{(T,K)}\|_{p-2}^{p-2} + 2\mu \leq \Phi(\ww_0)e^{O(1)}. 
\]
This implies that $\ww^{(T,K)} \leq O(1)\ww_0$ and $\|\ww^{(T,K)}\|_{\infty} \leq O(1)$. Therefore,
\[
\ff(\xxbar) \leq \left((O(1)RM\right)^p m.
\]
Now, using this bound on $\ff(\xxbar)$ and $\xxbar$ as a starting solution for Algorithm \ref{alg:eps}, we get our result by applying Theorem \ref{thm:main-eps}.
\end{proof}

\subsubsection{Logistic Regression}

\logisticRuntime*
\begin{proof}
From Lemma \ref{lem:runtime-qsc}, we get $\xxbar$ such that $\|\xxbar\|_{\infty} \leq RM\|\ww^{(T,K)}\|_{\infty}$. We now want to bound $\ff(\xxbar)$.
\[
\ff(\xxbar) = \sum_i \log(1 + e^{RM\ww^{(T,K)}_i}) \leq 2RM\sum_i \ww^{(T,K)}_i.
\]
We next note that for $\ww^{(T,K)}\geq \ww_0$,
\[
\Phi(\ww^{(T,K)}) = \sum_i \frac{e^{\ww^{(T,K)}_i}}{(1+e^{\ww^{(T,K)}_i})^2} \geq \Phi(\ww_0)e^{-O(1)}. 
\]
This implies that $\ww^{(T,K)} \leq O(1)\ww_0$ . Therefore,
\[
\ff(\xxbar) \leq O(R m).
\]
Now, using this bound on $\ff(\xxbar)$ and $\xxbar$ as a starting solution for Algorithm \ref{alg:eps}, we get our result by applying Theorem \ref{thm:main-eps}.
\end{proof}


\section{Energy Lemma}

\begin{lemma}\label{lem:r-dec}
Let $\Dtil = \arg\min_{\AA\xx=\cc} \xx^{\top}\PP^{\top}\RR\PP\xx$. Then one has for any $\rr$ and $\rr'$ such that $\rr'\leq \rr$,
\[
\Psi(\rr') \leq \Psi(\rr) - \frac{1}{2} \sum_i \left(1 - \frac{\rr'_i}{\rr_i}\right)\rr_i (\PP\Dtil)_i.
\]
\end{lemma}

\begin{proof}
\[
\Psi(\rr) = \min_{\AA\xx = \cc}  \xx^{\top}\PP^{\top}\RR\PP\xx.
\]
Constructing the Lagrangian and noting that strong duality holds,
\begin{align*}
\Psi(\rr) &= \min_{\xx}\max_{\yy} \quad  \xx^{\top}\PP^{\top}\RR\PP\xx + 2\yy^{\top}(\cc-\AA\xx)\\
& =\max_{\yy} \min_{\xx} \quad  \xx^{\top}\PP^{\top}\RR\PP\xx + 2\yy^{\top}(\cc-\AA\xx).
\end{align*}
Optimality conditions with respect to $\xx$ give us,
\[
2\PP^{\top}\RR\PP\xx^{\star} = 2\AA^{\top}\yy.
\]
Substituting this in $\Psi$ gives us,
\[
\Psi(\rr) = \max_{\yy}\quad 2\yy^{\top}\cc  - \yy^{\top}\AA \left(\PP^{\top}\RR\PP\right)^{-1} \AA^{\top}\yy.
\]
Optimality conditions with respect to $\yy$ now give us,
\[
2\cc  =  2 \AA \left(\PP^{\top}\RR\PP\right)^{-1} \AA^{\top} \yy^{\star},
\]
which upon re-substitution gives,
\[
\Psi(\rr) =  \cc^{\top}\left(\AA \left( \PP^{\top}\RR\PP\right)^{-1} \AA^{\top}\right)^{-1} \cc.
\]
We also note that 
\begin{equation}\label{eq:optimizer}
\xx^{\star} = \left(\PP^{\top}\RR\PP\right)^{-1}\AA^{\top}\left(\AA \left( \PP^{\top}\RR\PP\right)^{-1} \AA^{\top}\right)^{-1}\cc.
\end{equation}
We now want to see what happens when we change $\rr$. Let $\RR$ denote the diagonal matrix with entries $\rr$ and let $\RR' = \RR-\SS$, where $\SS$ is the diagonal matrix with the changes in the resistances. We will use the following version of the Sherman-Morrison-Woodbury formula multiple times,
\[
(\XX + \UU\CC\VV)^{-1} = \XX^{-1} - \XX^{-1}\UU(\CC^{-1} + \VV\XX^{-1}\UU)^{-1}\VV\XX^{-1}.
\]
We begin by applying the above formula for $\XX = \PP^{\top}\RR\PP$, $\CC = -\II$, $\UU = \PP^{\top}\SS^{1/2}$ and $\VV = \SS^{1/2}\PP$. We thus get,
\begin{multline}
\left(\PP^{\top}\RR'\PP\right)^{-1} = \left(\PP^{\top}\RR\PP\right)^{-1} +  \left( \PP^{\top}\RR\PP\right)^{-1}\PP^{\top}\SS^{1/2} \\
\left(\II - \SS^{1/2}\PP \left(\PP^{\top}\RR\PP\right)^{-1}\PP^{\top}\SS^{1/2}\right)^{-1}\SS^{1/2}\PP \left(\PP^{\top}\RR\PP\right)^{-1}.
\end{multline}
We next observe that, 
\[
\II - \SS^{1/2}\PP \left(\PP^{\top}\RR\PP\right)^{-1}\PP^{\top}\SS^{1/2} \preceq \II ,
\]
which gives us,
\begin{equation}
\left(\PP^{\top}\RR'\PP\right)^{-1} \succeq \left(\PP^{\top}\RR\PP\right)^{-1} + \left(\PP^{\top}\RR\PP\right)^{-1}\PP^{\top}\SS\PP \left(\PP^{\top}\RR\PP\right)^{-1}.
\end{equation}
This further implies,
\begin{equation}
\AA\left(\PP^{\top}\RR'\PP\right)^{-1}\AA^{\top} \succeq \AA\left(\PP^{\top}\RR\PP\right)^{-1}\AA^{\top} + \AA\left(\PP^{\top}\RR\PP\right)^{-1}\PP^{\top}\SS\PP \left(\PP^{\top}\RR\PP\right)^{-1}\AA^{\top}.
\end{equation}
We apply the Sherman-Morrison formula again for, $\XX =\AA\left(\PP^{\top}\RR\PP\right)^{-1}\AA^{\top}$, $\CC = \II$, $\UU =\AA\left(\PP^{\top}\RR\PP\right)^{-1}\PP^{\top}\SS^{1/2}$ and $\VV = \SS^{1/2}\PP \left(\PP^{\top}\RR\PP\right)^{-1}\AA^{\top}$. Let us look at the term $\CC^{-1} + \VV\XX^{-1}\UU$.
\begin{align*}
\CC^{-1} + \VV\XX^{-1}\UU&=   \II + \SS^{1/2}\PP \left(\PP^{\top}\RR\PP\right)^{-1}\AA^{\top}(\AA\left(\PP^{\top}\RR\PP\right)^{-1}\AA^{\top})^{-1}\AA\left(\PP^{\top}\RR\PP\right)^{-1}\PP^{\top}\SS^{1/2} \\
& \preceq   \II + \SS^{1/2}\PP \left(\PP^{\top}\RR\PP\right)^{-1}\PP^{\top}\SS^{1/2}\\
& \preceq \II + \SS^{1/2}\RR^{-1}\SS^{1/2}.
\end{align*}
Using this, we get,
\[
\left(\AA\left(\PP^{\top}\RR'\PP\right)^{-1}\AA^{\top}\right)^{-1} \preceq \XX^{-1} - \XX^{-1}\UU(\II + \SS^{1/2}\RR^{-1}\SS^{1/2})^{-1}\VV\XX^{-1},
\]
which on multiplying by $\cc^{\top}$ and $\cc$ gives,
\[
\Psi(\rr') \leq \Psi(\rr) -  \cc^{\top} \XX^{-1}\UU(\II + \SS^{1/2}\RR^{-1}\SS^{1/2})^{-1}\VV\XX^{-1} \cc.
\]
We note from Equation \eqref{eq:optimizer} that $\xx^{\star} = \left(\PP^{\top}\RR\PP\right)^{-1}\AA^{\top} \XX^{-1}\cc$.
We thus have,
\begin{align*}
\Psi(\rr') &\leq \Psi(\rr) -  \left(\xx^{\star}\right)^{\top} \PP^{\top}\SS^{1/2} (\II + \SS^{1/2}\RR^{-1}\SS^{1/2})^{-1}\SS^{1/2}\PP\xx^{\star}\\
& = \Psi(\rr) - \sum_e (\rr_e -\rr'_e)\left(1 + \frac{\rr_e-\rr'_e}{\rr_e}\right)^{-1}(\PP\xx^{\star})_e \\
& = \Psi(\rr) - \sum_e \left(\frac{\rr_e-\rr'_e}{2\rr_e-\rr'_e}\right)\rr_e(\PP\xx^{\star})_e\\
& \leq \Psi(\rr) - \frac{1}{2} \sum_e \left(\frac{\rr_e-\rr'_e}{\rr_e}\right)\rr_e(\PP\xx^{\star})_e
\end{align*}
Where the last line follows from the fact $2\rr_e - \rr'_e \leq 2\rr_e$.
\end{proof}

The next lemma is Lemma C.4 in \cite{ABKS21} which is included here for completeness.
\begin{lemma}
  \label{lem:ckmst:res-increase}
Let $\Dtil = \arg\min_{\AA\xx = c} \xx^{\top}\PP^{\top}R\PP\xx$. Then one has for any $\rr'$ and $\rr$ such that $\rr' \geq \rr$,
\[
{\Psi({\rr'})} \geq {\Psi\left({\rr}\right)} + \sum_{e}\left(1-\frac{\rr_e}{\rr'_e}\right) \rr_e (\PP\Dtil)_e^2.
\]
\end{lemma}
\begin{proof}
\[
\Psi(\rr) = \min_{\AA\xx = \cc}  \xx^{\top}\PP^{\top}\RR\PP\xx.
\]
Constructing the Lagrangian and noting that strong duality holds,
\begin{align*}
\Psi(\rr) &= \min_{\xx}\max_{\yy} \quad  \xx^{\top}\PP^{\top}\RR\PP\xx + 2\yy^{\top}(\cc-\AA\xx)\\
& =\max_{\yy} \min_{\xx} \quad \xx^{\top}\PP^{\top}\RR\PP\xx + 2\yy^{\top}(\cc-\AA\xx).
\end{align*}
Optimality conditions with respect to $\xx$ give us,
\[
 2\PP^{\top}\RR\PP\xx^{\star} = 2\AA^{\top}\yy.
\]
Substituting this in $\Psi$ gives us,
\[
\Psi(\rr) = \max_{\yy}\quad 2\yy^{\top}\cc  - \yy^{\top}\AA \left( \PP^{\top}\RR\PP\right)^{-1} \AA^{\top}\yy.
\]
Optimality conditions with respect to $\yy$ now give us,
\[
2\cc  =  2 \AA \left(\PP^{\top}\RR\PP\right)^{-1} \AA^{\top} \yy^{\star},
\]
which upon re-substitution gives,
\[
\Psi(\rr) =  \cc^{\top}\left(\AA \left( \PP^{\top}\RR\PP\right)^{-1} \AA^{\top}\right)^{-1} \cc.
\]
We also note that 
\begin{equation}\label{eq:optimizer2}
\xx^{\star} = \left(\PP^{\top}\RR\PP\right)^{-1}\AA^{\top}\left(\AA \left(\PP^{\top}\RR\PP\right)^{-1} \AA^{\top}\right)^{-1}\cc.
\end{equation}
We now want to see what happens when we change $\rr$. Let $\RR$ denote the diagonal matrix with entries $\rr$ and let $\RR' = \RR+\SS$, where $\SS$ is the diagonal matrix with the changes in the resistances. We will use the following version of the Sherman-Morrison-Woodbury formula multiple times,
\[
(\XX + \UU\CC\VV)^{-1} = \XX^{-1} - \XX^{-1}\UU(\CC^{-1} + \VV\XX^{-1}\UU)^{-1}\VV\XX^{-1}.
\]
We begin by applying the above formula for $\XX =\PP^{\top}\RR\PP$, $\CC = \II$, $\UU = \PP^{\top}\SS^{1/2}$ and $\VV = \SS^{1/2}\PP$. We thus get,
\begin{multline}
\left(\PP^{\top}\RR'\PP\right)^{-1} = \left(\PP^{\top}\RR\PP\right)^{-1} -  \left(\PP ^{\top}\RR\PP\right)^{-1}\PP^{\top}\SS^{1/2} \\
\left(\II + \SS^{1/2}\PP \left(\PP^{\top}\RR\PP\right)^{-1}\PP^{\top}\SS^{1/2}\right)^{-1}\SS^{1/2}\PP \left(\PP^{\top}\RR\PP\right)^{-1}.
\end{multline}
We next claim that
\[
\II + \SS^{1/2}\PP \left(\PP^{\top}\RR\PP\right)^{-1}\PP^{\top}\SS^{1/2} \preceq \II + \SS^{1/2}\RR^{-1}\SS^{1/2} ,
\]
which gives us,
\begin{multline}
\left( \PP^{\top}\RR'\PP\right)^{-1} \preceq \left( \PP^{\top}\RR\PP\right)^{-1} -  \\ \left( \PP^{\top}\RR\PP\right)^{-1}\PP^{\top}\SS^{1/2}(\II + \SS^{1/2}\RR^{-1}\SS^{1/2})^{-1}\SS^{1/2}\PP \left(\PP^{\top}\RR\PP\right)^{-1}.
\end{multline}
This further implies,
\begin{multline}
\AA\left( \PP^{\top}\RR'\PP\right)^{-1}\AA^{\top} \preceq \AA\left( \PP^{\top}\RR\PP\right)^{-1} \AA^{\top} -  \\ \AA\left(\PP^{\top}\RR\PP\right)^{-1}\PP^{\top}\SS^{1/2}(\II + \SS^{1/2}\RR^{-1}\SS^{1/2})^{-1}\SS^{1/2}\PP \left(\PP^{\top}\RR\PP\right)^{-1}\AA^{\top}.
\end{multline}
We apply the Sherman-Morrison formula again for, $\XX =\AA\left(\PP^{\top}\RR\PP\right)^{-1}\AA^{\top}$, $\CC = -(\II + \SS^{1/2}\RR^{-1}\SS^{1/2})^{-1}$, $\UU = \AA\left( \PP^{\top}\RR\PP\right)^{-1}\PP^{\top}\SS^{1/2}$ and $\VV = \SS^{1/2}\PP \left( \PP^{\top}\RR\PP\right)^{-1}\AA^{\top}$. Let us look at the term $\CC^{-1} + \VV\XX^{-1}\UU$.
\[
-\left(\CC^{-1} + \VV\XX^{-1}\UU\right)^{-1} =    \left(\II + \SS^{1/2}\RR^{-1}\SS^{1/2} -  \VV\XX^{-1}\UU\right)^{-1}  \succeq   (\II + \SS^{1/2}\RR^{-1}\SS^{1/2})^{-1}.
\]
Using this, we get,
\[
\left(\AA\left(\PP^{\top}\RR'\PP\right)^{-1}\AA^{\top}\right)^{-1} \succeq \XX^{-1} + \XX^{-1}\UU(\II + \SS^{1/2}\RR^{-1}\SS^{1/2})^{-1}\VV\XX^{-1},
\]
which on multiplying by $\cc^{\top}$ and $\cc$ gives,
\[
\Psi(\rr') \geq \Psi(\rr) +  \cc^{\top} \XX^{-1}\UU(\II + \SS^{1/2}\RR^{-1}\SS^{1/2})^{-1}\VV\XX^{-1} \cc.
\]
We note from Equation \eqref{eq:optimizer2} that $\xx^{\star} = \left(\PP^{\top}\RR\PP\right)^{-1}\AA^{\top} \XX^{-1}\cc$.
We thus have,
\begin{align*}
\Psi(\rr') &\geq \Psi(\rr) +  \left(\xx^{\star}\right)^{\top} \PP^{\top}\SS^{1/2} (\II + \SS^{1/2}\RR^{-1}\SS^{1/2})^{-1}\SS^{1/2}\PP\xx^{\star}\\
& = \Psi(\rr) + \sum_e \left(\frac{\rr'_e -\rr_e}{\rr'_e}\right)\rr_e (\PP\xx^{\star})_e .
\end{align*}
\end{proof}

\end{document}